\documentclass[a4paper,12pt left=30mm,right=25mm, top=2cm, bottom=2cm]{amsart}
\usepackage[style=alphabetic]{biblatex}
\usepackage[ansinew]{inputenc}
 
\usepackage{amsmath}
\usepackage{amsfonts}
\usepackage{amsthm}
\usepackage{amssymb}
\usepackage{graphicx}
\usepackage{tikz}
\usepackage{multirow}

\makeindex

\usepackage{soul} 
\sodef\so{}{.14em}{.4em plus.1em minus .1em}{.4em plus.1em minus .1em} 

\usepackage{aliascnt}
\usepackage{hyperref}
\hypersetup{
    colorlinks=false,
    pdfborder={0 0 0},}
\newcommand{\nref}[1]{\hyperref[#1]{\ref*{#1}}}

\usepackage{array} 
\usepackage[T1]{fontenc} 
\newcommand{\subtile}[1] 
{
	\vspace{-0.3cm}
	\begin{center}
 		{{\textsc{#1}}}\\
	\end{center}
	\vspace{0.1cm}
}

\newcommand{\B}{\mathcal{B}} 	
\newcommand{\g}{\mathfrak{g}} 	
\newcommand{\Z}{\mathbb{Z}}  	
\newcommand{\R}{\mathbb{R}}  	
\newcommand{\C}{\mathbb{C}}  	
\renewcommand{\k}{\Bbbk}  	
\newcommand{\F}{\mathbb{F}}  	
\newcommand{\D}{\mathbb{D}}  	
\newcommand{\Q}{\mathbb{Q}}  	
\renewcommand{\S}{\mathbb{S}}  	

\newcommand{\RR}{\mathcal{R}} 
\newcommand{\GG}{\mathcal{G}} 
\newcommand{\DD}{\mathcal{D}} 
\newcommand{\WW}{\mathcal{W}} 

\theoremstyle{plain}
\newtheorem{theorem}{Theorem}[section]
\newtheorem*{theoremX}{Theorem}

\newtheorem{corollary}[theorem]{Corollary}
\newtheorem*{corollaryX}{Corollary}

\newtheorem{definition}[theorem]{Definition}

\newtheorem{example}[theorem]{Example}

\newtheorem{lemma}[theorem]{Lemma}
\newtheorem*{lemmaX}{Lemma}

\newtheorem{remark}[theorem]{Remark}

\bibliography{SymplecticRootsystemsOverF2}{}

\newcommand{\hamburger}[1] 
{
  \pagestyle{empty}
  \vspace*{-3cm}
  \begin{center}
    \Large \bf
    #1
  \end{center}
  \vspace{0.5cm}
  \begin{center}	
    Simon Lentner \\
    Algebra And Number Theory (AZ), University Hamburg \\
    Bundesstraße 55, D-20146 Hamburg 
  \end{center}
  \vspace{-0.0cm}

}

\begin{document}

\hamburger{Root Systems In Finite Symplectic Vector Spaces}
\pagestyle{empty}
\begin{abstract}
  We study realizations of root systems in possibly degenerate
  symplectic vector spaces over finite fields, up to symplectic isomorphisms.
  The main result
  of this paper is the classification of such realizations for the field
  $\F_2$. Thereby, each root system requires a specific degree of degeneracy of
  the symplectic vector space. Our main  motivation for this paper is, that for
  each such realization of a root system we can construct a Nichols algebra
  over a nonabelian group.
\end{abstract}
  \makeatletter
  \@setabstract
  \makeatother

  \tableofcontents
  \newpage

\section{Introduction}
\subsection{Motivation and applications}

For a given $n\times n$ Cartan matrix with entries in $\Z$, a
root system of rank $n$ is generated by a basis of a $n$-dimensional euclidean
vector
space $V=\C^n$ (the simple roots) with the scalar products between all basis
elements prescribed by the Cartan matrix. The Cartan matrix is usually
visualized by a generalized Dynkin diagram. One demands
stability of the set of all roots under the action of the
Weyl/Coxeter group associated to the Cartan matrix. Root systems play among
others a prominent role in the theory of Lie algebras as well as Nichols
algebras, which appear naturally as Borel parts of 
quantum groups \cite{AS10}, such as $u_q(\g)$.\\

In our recent study of Nichols algebras over certain nonabelian groups $G$ of
nilpotency class $2$ in \cite{Len13} we started with a root system over $\C$
with given Cartan matrix for a Nichols algebra over an abelian group $G/[G,G]$.
Then, this Nichols algebra was extended to $G$ using an additional root system
structure on a symplectic vector space $V=G/G^2$ over the finite field
$\F_2$ with the same Cartan matrix. In this application, the symplectic form is
induced by the commutator map of $G$ and the prescribed Dynkin diagram is thus a
$G$-decorated commutativity graph. 


In the following article, we shall present a definition and classification of
\emph{symplectic root system} over the field $\F_2$. As every Cartan matrix can
only have entries $0,1\in\F_2$, it is sufficient to consider simply-laced Dynkin
diagrams and hence ordinary graphs. A symplectic root system over $\F_2$ is then
defined as a decoration of the Dynkin diagram graph by simple roots,
which are vectors in a (possibly degenerate) symplectic vector space $V=\F_2^n$,
such that the decorations of two nodes are (symplectic) orthogonal iff the
nodes are non-adjacent. The Coxeter group asociated to the
Dynkin diagram over $\C$ acts on the set of all roots by symplectic
isomorphisms. If the decorations form a basis of $V$, the
symplectic root system is called minimal.\\

We achieve a complete classification of symplectic root systems over $\F_2$ 
up to symplectic isomorphisms on arbitrary graphs. Especially we
clarify, which Dynkin diagram admits a symplectic root system for a given
nullity, i.e. the degree of degeneracy of the symplectic form and hence
the dimension of the nullspace $\dim(V^\perp)$. This nullity turns out to be
bounded by the coclique number of the graph. We call symplectic root systems
over nondegenerate symplectic vector spaces ``extraspecial'' as they correspond
to extraspecial groups in Section \nref{sec_CommutativityGraph}.\\

We conclude by giving credit to previous work:\\

Our notion of a symplectic root system has appeared already in
Lusztig's representation theory of finite Lie groups as a technical tool 
\cite{Lusz84} Chp. 9. It was also extensively studied in singularity theory
under the name \emph{vanishing lattices}\footnote{We thank Sergei Chmutov for
helpful comments.} by Wajnryb \cite{Wa80}, Chmutov \cite{Ch82}\cite{Ch83} and
Jansen \cite{Jan83}\cite{Jan85}. The possible groups
have been classified by \cite{Jan83} Thm. 4.8 and the number of isomorphism 
classes
of symplectic root systems for a given graph is reduced in \cite{Jan85} Thm. 7.5
to the case $\F_2$. To the best of our knowledge, our combinatorically derived
results are complementary and determine the explicit isomorphism classes over
$\F_2$, as well as the unique nullity, and apply also for non-minimal
symplectic root systems.

\subsection{Structure of the article}

We start with basic definitions in Section \nref{sec_Definition} and give first
examples and properties in Section \nref{sec_Properties}. Most importantly
we can prove already at this point a universal property of minimal
symplectic root systems, especially they are
unique up to isomorphism, as well as their existence. This shows that a
minimal symplectic root system of a given graph exists for precisely one
isomorphy type of symplectic vector spaces (nullity). However, this result does
not yet determine the nullity nor the isomorphy classes of the non-minimal
symplectic root systems, which is content of the remaining article.   

\begin{corollaryX}[Universal Property]
  Suppose $(f,V)$ and $(g,W)$ to be symplectic root systems on the same graph
  $\GG$ and assume moreover $(f,V)$ minimal. Then there exists a 
  homomorphism of symplectic root systems  $\phi:(f,V)\rightarrow (g,W)$.
  Especially two minimal symplectic root systems are always isomorphic. 
\end{corollaryX}

\begin{lemmaX}[Existence]
  For every graph $\GG$ there exists a minimal symplectic root system. By
  the universal property it is unique
  up to isomorphism.
\end{lemmaX}

We introduce a straightforward notion of quotients and find immediately:

\begin{corollaryX}
  For every graph $\GG$ there is up to isomorphism a unique minimal symplectic
  root system and all symplectic root system of $\GG$ are quotients thereof.
\end{corollaryX}

%
%

In Section \nref{sec_Restriction} we then consider the
restriction of a symplectic root systems to an induced subgraph $\GG\subset
\DD$ and derive bounds for the change in nullity of the symplectic root systems.
As an application we prove a bound on the nullity of a symplectic root system in
terms of the coclique number of the graph and briefly discuss the extremal cases
of the inequality (ADE- vs. complete graphs).\\

In the main Section \nref{sec_Extension} of this paper we introduce a
construction that extends a given minimal symplectic root system on a subgraph
by one node. 
\begin{theoremX}[Minimal Extensions]
  Let $\DD$ be a graph, $p\in\DD$ a node and $G=(g,W)$ a
  {minimal} symplectic root system of a spanning subgraph $\GG:=\DD-p$.
  Then there exists a unique minimal symplectic root system $F=(f,V)$ on $\DD$ 
  extending $G$.
\end{theoremX}
\begin{proof}
  The construction proceeds in the following steps for extending extraspecial,
  nullspace and finally arbitrary symplectic root systems. From the second
  step on, the extensions fall into two distinct cases yielding for $V$ either
  higher or lower nullity than $W$ according to the different cases in the
  Restriction Theorem \nref{thm_Restriction}.
\begin{itemize}
 \item In Section \nref{sec_extendingExtraspecial} we construct almost
  extraspecial extensions of	
  extraspecial symplectic root systems, i.e. $W$ nondegenerate.
  The proof uses the minimality of $G$ to express the indicator function
  $\lambda$ of a neighbourhood of $p$ in $\DD$ as a linear form
  $\tilde{\lambda}$. Then is uses the
  assumed nondegeneracy to construct a distinguished element
  $w_0\in W$ and a thereof the new decoration $f(p)\in V$.
  \item In Section \nref{sec_extendingNullspace} we classify in contrast
  extensions of symplectic root systems consisting only of nullspace
  $W=W^\perp$. Especially $\GG$ is totally disconnected.
  Thereby we need surprisingly the choice of an additional nondegenerate
  symmetric bilinear form
  $(,)$ on the nullspace regarded as  vector space over $\F_2$. As before we
  construct a linear form $\tilde{\lambda}$, but
  we yield two cases:
  Either $\tilde{\lambda}=0$, then $\DD$ is again totally disconnected,
  $V=V^\perp$ is extended by yet another nullvector and has higher nullity then
  $W$. For $\tilde{\lambda}\neq0$ we decompose 
  $W=\ker(\tilde{\lambda})\oplus x\k$ and extend $x$ by
  the to-be-definied decoration  $f(p)=y$ to a new hyperbolic plane $H_1$.
  Especially $V$ then has lower nullity than $W$
  \item In Section \nref{sec_extendingArbitrary} we combine the preceeding
results and achieve the final classification result for extending arbitrary
sympletic root systems by one point. The crucial ingredient is an artificial 
nondegenerate, symmetric bilinear form $\langle\langle,\rangle\rangle$ extending
the symplectic form. Thereby we effectively write the neighbourhood of the
new point as a symmetric difference of two graphs obtained by the two previous
methods.
  \item In Section \nref{sec_doubleExtension} we give an additional very
explicit formula for 2-point-extensions of a given extraspecial symplectic
root system. The criterium avoids the use of the artificial mixed
bilinear form and provides a nice characterization in terms of the two
1-point extensions and their possible interaction. This is used in the example
calculations for ADE-type in Section \nref{sec_CartanType}.
\end{itemize}
\end{proof}
%
%

Note
that determining the nullity for a given graph is tedious and usually requires
to successively apply the extension theorem. As an example, in Section
\nref{sec_CartanType} we determine explicitly decoration and
nullity for all symplectic root systems corresponding to Dynkin diagrams
of finite Cartan type, i.e. in the ADE family. We start with an induction on
$A_{2n-2}\rightarrow A_{2n}$, which turns out to yield minimal
symplectic root systems of extraspecial type. Then we extend by explicit nodes
to reach the other diagrams with higher nullity. We indeed find that Cartan type
root systems typically require the smallest possible nullity ($0$ or $1$),
except $D_{2n}$ needs $2$. We then determine all non-minimal quotients up to 
Dynkin diagram automorphisms.\\

As further topics, we first connect in Section \nref{sec_Weyl} the notion of
symplectic root systems to the well-known notion over $\C$. We describe, how
a symplectic decoration of the graph as considered here can be additively
extended to a full root systems and prove that the Coxeter group accociated to
the prescribed Cartan matrix (e.g. the Weyl group) acts on this natural set of
symplectic vectors by symplectic isomorphisms.\\

Section \nref{sec_CommutativityGraph} explains the application of symplectic
root systems to Nichols algebras. In 
\cite{Len13}  we constructed the first Nichols algebras of rank
$>2$ over nonabelian groups $G$. The construction starts with a known
finite-dimensional Nichols algebra over an abelian group $\Gamma$ of
simply-laced Cartan type with a diagram automorphism $\Z_2$. Then, a symplectic
root system over the field $\F_2$ is used to construct a new link-indecomposable
finite-dimensional covering Nichols algebra over a central extension
$\Z_2\rightarrow G \rightarrow \Gamma$ and again over $\C$. \\

Here, the symplectic vector space over $\F_2$ is
$V:=\Gamma/\Gamma^2$ with the symplectic form induced by the commutator map
on $G$. The symplectic root systems then provides a basis of $V$ adapted to the
existing Dynkin diagram. Especially the nullity of $V$ corresponds to a
specific size of the center $Z(G)$. We hope the present classification
will enable the classification of diagonal Nichols algebras over
nilpotent groups.

\section{Main Definition}\label{sec_Definition}
Throughout this article, we assume all graphs to be finite and
all vector spaces to be finite-dimensional.

\begin{definition}
  The following unusually general notion is custom in the theory of
  $p$-groups, see e.g. \cite{Hup83} paragraph II.9 (p. 215): A (possibly
  degenerate) symplectic vector
  space is a vector space $V$ over a field $\k$ with a (possibly degenerate)
  bilinear form 
  $$\langle ,\rangle:\; V\times V\rightarrow \k$$
  that is alternating: 
  $$\langle v,v\rangle=0$$
  Note that alternating forms are always skew-symmetric, but for
  $\mbox{char}(\k)=2$ skew-symmetric is equivalent to symmetric and does not
  imply alternating. The {nullspace} (or radical) of $V$ is
  defined as 
  $$V^\perp:=\{v\in V\;|\;\forall_{w\in V}\; \langle v,w\rangle=0\}$$
  Note that $W^\perp$ for a subspace denotes the radical $\subset W$, not the
  complement $\subset V$. 
\end{definition}

\begin{theorem}\label{thm_SymplecticBasis}(\cite{Hup83} Satz II.9.6 (p.217))
Every (finite-dimensional) symplectic vector space $V,\langle,\rangle$ possesses
a {symplectic basis}\index{Symplectic
Basis} $\{x_1,\ldots x_n,y_1,\ldots y_n,z_1,\ldots z_k\}$
consisting of mutually orthogonal {nullvectors} $z_i\in V^\perp$ and hyperbolic
planes $H_i=x_i\k\oplus y_i\k,\quad \langle x_i,y_i\rangle=1$
\begin{align*}
 V\cong & \left(x_1\k\oplus y_1\k\right)\oplus^\perp 
	  \left(x_2\k\oplus y_2\k\right)\oplus^\perp\cdots
	  \oplus^\perp z_1\k \oplus^\perp z_2\k\oplus^\perp\cdots
\end{align*}
\end{theorem}

\begin{definition}  Denote by $k=dim V^\perp$ the {nullity}, the number of
  nullvectors $z_i$ in a basis. Further denote by $n$ the conullity, the number
  of hyperbolic planes $H_i=x_i\k\oplus y_i\k$ generating $V$. Altogether
  $2n+k=\dim V$. By Theorem
  \nref{thm_SymplecticBasis} the {type} $(n,k)$ precisely characterizes the
  isomorphism type of a symplectic vector space $V$.  
  \begin{itemize}   
  \item We call $V$ {extraspecial}, iff the form is nondegenerate $V^\perp=0$,
      i.e. the $V$ is of type $(n,0)$ and necessarily of even dimension $2n$.
      The name is chosen as the associated $p$-groups are usually called
      extraspecial, see Section \nref{sec_CommutativityGraph}.
   \item We call $V$ {almost extrapecial}, iff $\dim V^\perp=1$, i.e. $V$ is
      of type $(n,1)$  and necessarily of odd dimension $2n+1$. The name
      is chosen as the associated $p$-groups are usually called almost
      extraspecial, see Section \nref{sec_CommutativityGraph}.
  \end{itemize}
\end{definition}

The key notion of a symplectic root system over $\F_2$ will now be given purely
in terms of graph theory, as the Dynkin diagram is always simply-laced in
characteristic $2$. Note that we yet have no satisfying general definition for
arbitrary characteristic.

\begin{definition}
  A symplectic root system $F=(f,V)$ over the field $\k=\F_2$
  on a (finite) graph $\DD$ consists of a (finite-dimensional)
  symplectic vector space $V$ over the field $\F_2$ together with a node
  decoration $f:\DD\rightarrow V$  such that
  \begin{itemize}
    \item The image $f(\DD)$ generates $V$ as a vector space.
    \item Two nodes $p,q\in \DD$ are adjacent in the graph $\DD$ iff their
      $f$-decorations in $V$ are not orthogonal 
      $\langle f(\alpha),f(\beta)\rangle \neq 0$ (then already $=1_{\F_2}$),
  \end{itemize}
   We call $F=(f,V)$ of {type} $(n,k)$, iff the symplectic vector space $V$
  is of type $(n,k)$. We call $F$ {minimal} iff the decorations
  $\{f(\alpha)\}_{\alpha\in\DD}$ form a basis of $V$.
 \end{definition}

\section{First Properties And Examples}\label{sec_Properties}

\begin{example}\label{exm_point}
  Let $\GG=A_1=\{p\}$ be an isolated point. Because the node decorations
  $f(p)$
  have to generate $V$, the only symplectic root systems $F=(f,V)$ are
  \begin{itemize}
    \item $V=z\F_2$ with $f(p)=z$ of type $(0,1)$, which is minimal.
    \item $V=\{0\}$ with $f(p)=0$ of type $(0,0)$, which is not minimal.
  \end{itemize}
\end{example}

\begin{lemma}\label{lm_Disconnected}
  Let $\GG=\GG_1\cup\GG_2$ be a {disconnected union} of subgraphs. Then any
  {minimal} symplectic root system  $F=(f,V)$ decomposes as
  $V=V_1\oplus^\perp V_2$ with $f(\GG_{i})\subset  V_{i}$ inducing minimal
  symplectic root system on each subgraphs. The orthogonal sum especially
  implies, that if these smaller symplectic root system are of type
  $(n_1,k_1)$ resp. $(n_2,k_2)$,  then $F$ is of type $(n_1+n_2,k_1+k_2)$.
\end{lemma}
\begin{proof}
  Because the decorations $\{f(p)\}_{p\in\GG}$ by the assumed minimality of
  $F$ form a basis of $V$, we may decompose $V$ as direct sum of subspaces
  $V_i$ with basis $\{f(p)\}_{p\in\GG_i}$:
  $$V=V_1\oplus V_2$$
  $\GG_1,\;\GG_2$ were assumed to be mutually disconnected, so the defining
  property of symplectic root systems implies
  $$\forall_{p\in\GG_1\;q\in\GG_2}\; \langle f(p),f(q)\rangle=0$$
  As the sets $\{f(p)\}_{p\in\GG_i}$ generate respectively $V_i$ the
  sum is orthogonal as asserted.
\end{proof}

\begin{example}
  Let $\GG$ be a {totally disconnected graph}, then the preceeding lemma
  shows each minimal symplectic root system to be of type $(0,|\GG|)$ as $V$
  decomposes orthogonally into $1$-dimensional nullspaces for each isolated
  point as in example \nref{exm_point}. 
\end{example}

If the graph contains proper edges, a pure nullspace
$V=V^\perp$ will not suffice:

\begin{example}\label{exm_A2}
  Let $\GG=A_2=\{p,q\}$ be two connected points. The decorations
  $f(p),f(q)$ have to generate $V$, thus it can have dimension at most $2$.
  Pure nullspace cases can be discarded, because $\langle f(p),f(q)\rangle=0$
  contrary to the assumed edge $pq$.\\

  Hence the only remaining possibility is a single hyperbolic plane
  $V=x\F_2\oplus^{\not\perp} y\F_2$, i.e. $V$ of type $(1,0)$.  Here, indeed we
  yield a minimal symplectic root system  $f(p)=x$ and $f(q)=y$ with $\langle
  f(p),f(q)\rangle=1$. Note that there are many other possibilities, e.g.
  $f(p)=x+y$ and $f(q)=y$, but these choices obviously just differ by a linear
  isomorphism that preserves the symplectic form.
\end{example}

The last example shows, that one has to classify symplectic root system
according to some isomorphism criterion.

\begin{definition}[Homomorphisms of symplectic vector spaces]
  Let $V,W$ be (possibly degenerate) symplectic vector spaces, then we call a
  linear map $\phi:V\rightarrow W$ a symplectic homomorphism, iff
  $$\forall_{v,w\in V}\;\langle v,w\rangle_V
    =\langle\phi(v),\phi(w)\rangle_W$$
  If moreover $\phi$ is bijective, we call $\phi$ symplectic isomorphism. Note
  that if $V$ is degenerate, a symplectic homomorphism needs not to be
  bijective.
  Rather, $\phi$ might possess a kernel $\ker(\phi)\subset V^\perp$.
  The conullies of $V$ and $\text{Im}(\phi)\subset W$ coincide, but
  the nullity of $V$ might be higher.
\end{definition}

\begin{definition}[Morphisms]
  Let $\GG$ be a fixed graph and $F=(f,V),\;G=(g,W)$ symplectic root systems of
  $\GG$. A {symplectic root system homomorphism} $\phi:F\rightarrow G$ is a
  linear  map $\phi:V\rightarrow W$ that  intertwines the decorations:
  $\phi\circ f=g$. If moreover $\phi$ is bijective, we call $\phi$ a 
  {symplectic root system isomorphism}. 
\end{definition}
\begin{remark}
    Note that we require the graph $\DD,\GG$ to be fixed sets.
    Hence a graph automorphism on $\DD$
    may interchange non-isomorphic symplectic root system in this
    definition. For nontrivial examples of ADE-type see Remark
    \nref{rem_ActionOnCartan}.
\end{remark}
 
The assertion of $\phi\circ f=g$ for symplectic root
system morphisms already proves $\phi$ to be a
a symplectic homomorphism, as the defining property of a symplectic root system
$(f,V)$ already fixes the entire symplectic form on $V$:

\begin{lemma}\label{lm_firstUniversal}
  Suppose $\phi:(f,V)\rightarrow (g,W)$ to be a homomorphism between symplectic
  root systems. Then $\phi$ is already a surjective symplectic homomorphism
  $\langle v,w\rangle_V=\langle\phi(v),\phi(w)\rangle_W$.
\end{lemma}
\begin{proof}
    We calculate that $\phi$ preserves the symplectic form. First, the defining
    property of the symplectic root systems $(f,V)$ and $(g,W)$ on the same
    graph already fixes the value of the symplectic form on the decorations: 
    \begin{align*}
    \forall_{p,q\in \GG}\;\; 
    \langle f(p),f(q)\rangle
    &=\langle g(p),g(q)\rangle =0_{\F_2},1_{\F_2}
  \end{align*}  
  Because we assumed $\phi$ to be a homomorphism of symplectic root systems 
  $\phi\circ f=g$
  $$\forall_{p,q\in \GG}\;\;\langle f(p),f(q)\rangle =
  \langle\phi(f(p)),\phi(f(q))\rangle$$
  so $\phi$ is a symplectic homomorphism $V\to W$. Moreover, by definition of a
  symplectic root system the images $f(p)$ generate $V$, hence $\phi$ is
  necessarily surjective.
\end{proof}

Because for a minimal symplectic root system $(f,V)$ the decorations
$\left\{f(p)\right\}_{p\in \GG}$ even form by definition a basis of $V$, we
always find in this case a unique linear map $\phi:V\rightarrow W$ with
$g=\phi\circ f$. 

\begin{corollary}[Universal Property]\label{cor_AlwaysHomomorphism}
  Suppose $(f,V)$ and $(g,W)$ to be symplectic root systems on the same graph
  $\GG$ and assume moreover $(f,V)$ minimal. Then there exists a 
  unique homomorphism of symplectic root systems  $\phi:(f,V)\rightarrow (g,W)$.
  Especially two minimal symplectic root systems are always isomorphic.
\end{corollary}
On the other hand, since $\GG$ is supposed finite: If $(g,W)$ is minimal, then
any symplectic root system homomorphism $\phi:(f,V)\rightarrow (g,W)$ is an
isomorphism.\\

By inducing a symplectic form on the formal vector space generated by the nodes
of $\GG$ we find morover\footnote{We thank the referee for pointing out this
significantly easier approach}:

\begin{lemma}[Existence]\label{lmNEW_Existence}
  For every graph $\GG$ there exists a minimal symplectic root system. By
  the universal property (Corollary \nref{cor_AlwaysHomomorphism}) it is unique
  up to isomorphism.
\end{lemma}
\begin{proof}
  Consider the vector space $V$ over $\F_2$ generated by a basis $\{v_p\;|\;p\in
  \GG\}$ and the symplectic form defined on this basis by 
  $$\langle v_p,v_q \rangle=\begin{cases}
                             1_{\F_2}, & \text{if $p\neq q$ adjacent}\\
			     0_{\F_2}, & \text{else}
                            \end{cases}$$
  The bilinear form $\langle,\rangle$ is symmetric and we check it is indeed an
  alternating form over $\F_2$:
  $$\langle\sum_{p\in\GG} a_pv_p,\sum_{p\in\GG} a_pv_p\rangle
  =\sum_{p\in\GG} a_p^2\langle v_p,v_p\rangle
  +2\sum_{p\neq q\in \GG} a_pa_q\langle v_p,v_q\rangle=0$$
  Moreover, it is clear that by definition $F:=(f,V)$ with $f(p)=v_p$ is a
  symplectic  root system.
\end{proof}
%
%

Let conversely be $\phi:\;V\rightarrow W$ be a {surjective symplectic
homomorphism}, i.e
$\phi$ surjective and 
$$\langle v,w\rangle_V=\langle
  \phi(v),\phi(w)\rangle_W$$
For a given symplectic root system $F=(f,V)$ on a graph $\GG$ we may consider
the pair $G:=(\phi\circ f,W)$ and verify for $G$ the defining property of a
symplectic root system on the same underlying graph $\GG$. First we check  
$$\forall_{p,q\in \GG}\;\;\langle (\phi\circ f)(p),(\phi\circ f)(q)\rangle_W 
  =\langle f(p),f(q)\rangle_V=0_{\F_2},1_{\F_2}$$
Secondly by definition the symplectic root system $F$ requires $\text{Im}(f)$ to
generate $V$ and thus by the assumed surjectivity of $\phi$ the images of
$\phi\circ f$ generate again $W$. Hence we have proven that $G$ is
indeed a new symplectic root system. Especially by dimensionality
reasons, the new root system cannot be minimal unless $\phi$ is a symplectic
isomorphism; in which case the symplectic root systems $G$ and $F$ are
isomorphic.

\begin{definition}[Quotient]
  Let $(f,V)$ be a symplectic root system on a graph $\GG$ and
  $\phi:\;V\rightarrow W$ be a {surjective symplectic homomorphism}, then we
  define the
  quotient symplectic root system on the same graph by $(\phi\circ f,W)$. 
  Thereby $\phi$ becomes a homomorphism between the symplectic root
  systems. 
\end{definition}

In view of the universal property and the existence of minimal
symplectic root systems by Corollary \nref{cor_AlwaysHomomorphism} and Lemma
\nref{lmNEW_Existence}) we have altogether: 

\begin{corollary}\label{corNEW_allSRS}
  For every graph $\GG$ there is up to isomorphism a unique minimal symplectic
  root system and all symplectic root system of $\GG$ are quotients thereof.
\end{corollary}

\section{A Restriction Theorem}\label{sec_Restriction}

\begin{definition}
  Let $F=(f,V)$ be a symplectic root system of a graph $\DD$ and $\GG \subset
  \DD$ the induced subgraph of a subset of vertices of $\DD$. Then we define the
  restriction $F|_{\GG}$ to be $(f|_{\GG},W)$, where $f$ is restricted to $\GG$
  and $W\subset V$ is the subspace generated by the image $f(\GG)$. The
  restriction is clearly a symplectic root system of $\GG$.
\end{definition}

%

\begin{lemma}[Orthogonal Projection]\label{lm_OrthogonalProjection}
  This is a generalization of orthogonal projection to the degenerate case
  (and probably not new):
  Suppose $W\subset V$ to be symplectic vector spaces and suppose a given $v\in
  V$ such that  $v\perp W^\perp$. Then there exists a decomposition $v=v_0+v_W$
  with $v_W\in  W$ and $v_0\perp W$. Moreover, all
  such decompositions are of the form $v=(v_0+t)+(v_W-t)$ with some $t\in
  W^\perp$.
\end{lemma}
\begin{proof}
  Denote by $x_i,y_i,z_j$ any symplectic basis of $W$ (Theorem
  \nref{thm_SymplecticBasis}) and define
  $$v_W:=\sum_k \left(\langle v,y_k \rangle x_k
  -\langle v,x_k \rangle y_k\right)$$
  Then as claimed $v_W\in W$ by construction. Moreover we define 
  $v_0:=v-v_W$ and check on the  symplectic basis that as asserted 
  $v_0\perp W$:	
  \begin{align*}
    \langle v_0,x_i\rangle
      &=\langle v-v_W,x_i\rangle\\
      &=\langle v,x_i\rangle 
	- \sum_k \left(\langle v,y_k \rangle \langle x_k,x_i\rangle 
	-\langle v,x_k \rangle\langle y_k,x_i\rangle \right)\\
      &=\langle v,x_i\rangle 
	+ \langle v,x_i \rangle\langle y_i,x_i\rangle=0
  \end{align*}
  The same calculation proves $v_0\perp y_i$. Finally we have by construction
  $v_0\in v+W$ and hence $v_0\perp W^\perp$ by the additional assumption
  $v\perp W^\perp$.\\

  For the last claim, assume that $v=v_0+v_W=v_0'+v_W'$
  with $v_0,v_0'\in W$ and $v_W,v_W'\perp W$. Then $t:=v_0'-v_0=-(v_W'-v_W)$ is
  contained in $W$ and $t\perp W$, hence $t\in W^\perp$.
\end{proof}

\begin{theorem}[Restriction]\label{thm_Restriction}
  Let $F=(f,V)$ be a symplectic root system of type $(n,k)$ on a graph $\DD$, 
  let $p\in\DD$ be any node and denote by $\GG$ the induced subgraph on $\DD-p$.
  Then the restricted symplectic root system
  $F|_{\GG}=(f|_{\GG},W)$ is either
  \begin{enumerate}
    \item of same type $(n,k)$ and $F$ was not minimal
    \item of type $(n,k-1)$, especially $F$ was not extraspecial
    \item of type $(n-1,k+1)$, especially $F|_{\GG}$ is not extraspecial
  \end{enumerate}
\end{theorem}

\begin{proof}
 Consider the subspace $W\subset V$ generated by the image $f(\GG)$.
  Either of the following cases applies:
\begin{enumerate}
  \item $V=W$, then $F|_{\GG}$ has the same type as $F$. Further, $F$ could
    not have been minimal, because the decoration $f(p)\in W$ was linear
    dependent.
  \item $V=f(p)\F_2\oplus W$. We denote by $(n',k')$ the type of $W$ and aim to
    determine the two remaining cases as $k'=k\pm 1$. Then $V$ contains at
    least $n'$ mutually orthogonal hyperbolic planes, hence $n\geq n'$, and the
    nullspace is $\dim(V^\perp)\geq k'-1$ (it depends on whether $f(p)\perp
    W^\perp$). This implies by $\dim(V)=2\cdot n+k$ and
    $\dim(W)=2\cdot n'+k'$ that the type $(n',k')$ is as claimed.
%
%
%
\end{enumerate}
\end{proof}

\begin{corollary} 
 For $\gamma$ the size of a maximal coclique $\Gamma$ in a graph $\GG$
  and $(f,V)$ a symplectic root system of $\GG$ of type $(n,k)$ we get the
following bound:
  $$n\leq |\GG-\Gamma|=|\GG|-\gamma$$
\end{corollary}
\begin{proof}
  We perform induction along $|\GG-\Gamma|$. For $\GG=\Gamma$ the graph is
  totally disconnected, hence obviously (see Lemma \nref{lm_Disconnected})
  any symplectic root system is of type $n=0$ and the bound holds with
  equality.\\

  Now we proceed with the induction step: Suppose $(f,V)$ to be a symplectic
  root system on some graph
  $\GG$ with a maximal coclique $\Gamma\subsetneq\GG$. Choose any $p\in
  \GG-\Gamma$ and consider the restriction $(f|_{\GG-p},W)$ of type $(n',k')$.
  By induction, 
  $$n'\leq |\GG-p|-\gamma'= |\GG|-1-\gamma$$
  because $\Gamma$ is also a coclique in $\GG-p$, which is maximal, hence
  $\gamma'= \gamma$. Now by Theorem \nref{thm_Restriction}
  we have $n=n'$ or $n=n'+1$, hence $n\leq n'+1$ and as asserted
  $$n\leq n'+1\leq |\GG|-1-\gamma +1=|\GG|-\gamma$$
\end{proof}

Note that the bound in the corollary is met for all Cartan type graphs in
Section \nref{sec_CartanType} (even $D_{2n}$). On the other
hand, the complete graphs
$\GG=\mathcal{K}_N$ provide examples with $\gamma=1$ (the maximal possible value
on the right hand side) but still (almost-) extraspecial $n\approx|\GG|/2$ --
the minimum possible values on the left hand side.

\section{An Extension Theorem}\label{sec_Extension}

\begin{definition}
  Let $\DD$ be a graph and $\GG\subset \DD$ a subgraph induced by a
  subset of vertices in $\DD$. A symplectic root system $F=(f,V)$ on
  $\DD$ is called {extension} of a given symplectic root system $G=(g,W)$ on
  $\GG$, if the restriction $F|_{\GG}=G$.
\end{definition}

\subsection{Extending Extraspecial Symplectic Root Systems}
\label{sec_extendingExtraspecial}

\begin{theorem}[Extension of Extraspecials]\label{thm_ExtensionExtraspecial}
  Let $\DD$ be a graph, $p\in\DD$ a node and $G=(g,W)$ a
  {minimal extraspecial} symplectic root system of $\GG:=\DD-p$. Then there
  exists an almost extraspecial minimal extension $F=(f,V)$ of $G$ to
  $\DD$.
\end{theorem}
\begin{proof}
  To construct $F$, take $V=z\F_2\oplus^\perp W$, which is an almost
  extraspecial symplectic vector space, and $f|_W=g$ extending $G$ as demanded.
  To choose the new decoration $f(p)$, we consider the
  indicator function of the neighbourhood of $p$ in $\GG$:
  $$\lambda:\;\GG\rightarrow \F_2 \qquad 
    \forall_{q\in \GG}\;\lambda(q)=1\;:\Leftrightarrow\; pq\in Edges(\DD)$$
  Because $G$ was assumed minimal i.e. $g(\GG)$ is a basis of $W$, there is a
  unique linear form $\tilde{\lambda}$ on $W$ such that
  $$\tilde{\lambda}:\;W\rightarrow \F_2\qquad
    \forall_{q\in \GG}\;\lambda(q)=\tilde{\lambda}(g(q))$$
  Because we assumed $G$ extraspecial, i.e. $W$ nondegenerate, there is a
  unique element 
  $$w_0\in W\qquad 
    \forall_{w\in W}\;\langle w_0,w\rangle=\tilde{\lambda}(w)$$
  The choice $f(p):=w_0+z$ then turns $F=(f,V)$ into a symplectic
  root system, as
  $$\forall_{q\in \GG}\;\langle f(p),f(q)\rangle=\langle w_0,g(q)\rangle
    =\tilde{\lambda}(g(q))=\lambda(q)$$
  while for elements $q,q'\in \GG$ the condition was already satisfied in $G$.
  Moreover, as the $f(q)=g(q)$ for $q\in\GG$ already formed a basis of $W$ by
  assumption of minimality, together with $f(p)\in z+W$ we get a basis of $V$.
  Hence we constructed an almost extraspecial minimal symplectic root system
  on $\DD$. 
\end{proof}

\subsection{Extending Nullspace Symplectic 
Root Systems}\label{sec_extendingNullspace}

The proof of the last extension theorem crucially relies on the
nondegeneracy of the symplectic form. However, for a degenerate
vector space $W$ of type $(n,k)$, generally only few edge-configurations to the
new point $p$ can be constructed by simply using an existing vector in $w_0\in
W$ together with a new nullvector $z$ as above.\\

 This can be seen already by
counting, as there are $|W/W^\perp|=2^{2n}$ possible $w_0$ in contrast to
$2^{2n+k}$ possible edge configurations. Note that these cases correspond to the
restriction case 2 in Theorem \nref{thm_Restriction}.
The remaining $2^{2n}\left(2^k-1\right)$ configurations require, in addition to
the choice of an element $w_0$, a new symplectic base-pair in $V$
from a former nullvector in $W$ causing additional new edges -- corresponding
to the restriction case 3 in Theorem \nref{thm_Restriction}.\\

Before turning to the question of extending an arbitrary minimal symplectic root
system, let us look at the opposite extreme case: The extension of a symplectic
root system of type $(0,k)$, i.e. on a pure nullspace $W=W^\perp$. Because of
the Restriction Theorem \nref{thm_Restriction}, the extension $V$ can only have
type $(0,k+1)$ (=new nullvector) or $(1,k-1)$ (=new hyperbolic plane). We
show that both cases can be constructed depending on the aspired
graph extension. The following proof models similarly to the proof of Theorem
\nref{thm_ExtensionExtraspecial} the indicator function $\lambda$ of a
neighbourhood of $p$ by a linear form $\tilde{\lambda}$ via an artificially
chosen nondegenerate symmetric bilinear form on $W^\perp$. Then in the
nontrivial
case $\tilde{\lambda}\neq 0$ one constructs $V$ by extending the kernel of
$\tilde{\lambda}$ by a new hyperbolic plane $H_1=x_1\k\oplus y_1\k$.

\begin{theorem}[Extension of Nullspaces]\label{thm_ExtensionNullspace}
  Let $G=(g,W)$ be a minimal symplectic root system of type $(0,k)$ on
  a graph $\GG$, i.e. $W=W^\perp$ a pure nullspace. Let
  $\DD$ be a graph with
  $p\in\DD$ a node, such that $\GG:=\DD-p$. Then there exists a minimal
  extension $F=(f,V)$ of $G$ to $\DD$ of
  either one of the following types:
  \begin{itemize}
    \item $(0,k+1)$, if $\DD$ is completely disconnected.
    \item $(1,k-1)$, if $\DD$ contains at least one edge.
  \end{itemize}
  The extension $F$ is unique up to isomorphism by Corollary
  \nref{corNEW_allSRS}.
\end{theorem}
Note that in this Theorem the graph
$\GG$ has to be totally disconnected by definition and $\DD$ is a star
graph with center $p$ connected to a subset of $\GG$, the neighbourhood of
$p\in\DD$. The first case of the Theorem is then if the neighbourhood of $p$ in
$\DD$ is empty.
\begin{proof}
  The first case is a trivial construction, so let's turn to the second case:
  In order to produce edges at all (see Example \nref{exm_A2}), $V$ of
  dimension $1+k$ cannot be a
  pure nullspace and thus of type $(n,1+k-2n)$ for $n>0$. The vice-versa
  restriction to $\GG=\DD-p$ is the assumed $G$ of type $(0,k)$, so by
  Theorem \nref{thm_Restriction} only $n=1$ an hence type $(1,k-1)$ remains:
  This is case 2 in the cited Restriction Theorem: $V$ possesses one
  hyperbolic plane $H_1=x\k\oplus y\k$, that is separated by the restriction to
  $\GG$ and $W$.\\

  We may explicitly construct such a $V$ as follows: Consider again  the
  indicator function $\lambda$ of the neighbourhood of $p$  in $\DD$ (such a 
  neighbourhood is subset of $\GG$):
  $$\lambda:\;\GG\rightarrow \F_2 \qquad 
    \forall_{q\in \GG}\;\lambda(q)=1\;:\Leftrightarrow\; pq\in Edges(\DD)$$
  By assumption of the second case of the assertion, $\DD$ not totally
  disconnected, so $\lambda\neq 0$. Because $G$ was assumed minimal i.e.
  $g(\GG)$ is a basis of $W$, there is a unique linear form $\tilde{\lambda}\neq
  0$ on $W$ such that
  $$\tilde{\lambda}:\;W\rightarrow \F_2\qquad
    \forall_{q\in \GG}\;\lambda(q)=\tilde{\lambda}(g(q))$$
  Choose an element $x\in W$ with $\tilde{\lambda}(x)=1_{\F_2}$, which is
  possible by $\tilde{\lambda}\neq 0$. Such an $x$ can be used to project
  $W\rightarrow \ker(\tilde{\lambda})$ by $w\mapsto
  w-\tilde{\lambda}(w)\cdot x$. \\

  With these preperations we define:
  \begin{align*}
    V
    &:= \underbrace{\ker\left(\tilde{\lambda}\right)}_{V^\perp}\oplus^\perp
      \left( x\k\oplus y\k\right)=W\oplus y\k
    \qquad \langle x,y\rangle:=1\\
    f(p)
    &:=y\\
    f(q)
    &:=g(q)\\
    &:=\underbrace{g(q)-\tilde{\lambda}\left(g(q)\right)\cdot x}
      _{\in \ker\left(\tilde{\lambda}\right)}
      +\tilde{\lambda}\left(g(q)\right)\cdot x
    \qquad \forall q\in \GG
  \end{align*}
  Note that the vice-versa restriction to $\GG$ is equal to $G$.
  The pair $(f,V)$ defined this way satisfies the axioms of a symplectic
  root system. This can be seen for node pairs $p,q$ involving the new point $p$
  by
  \begin{align*}
    \forall_{q\in \GG}\;\langle f(p),f(q)\rangle
    &=\langle y,\underbrace{g(q)-\tilde{\lambda}\left(g(q)\right)\cdot x}
      _{\in \ker\left(\tilde{\lambda}\right)\subset V^\perp}
      +\tilde{\lambda}\left(g(q)\right)\cdot x\rangle\\
    &=\langle y,\tilde{\lambda}\left(g(q)\right)\cdot x\rangle\\
    &=\tilde{\lambda}\left(g(q)\right)=\lambda(q) 
  \end{align*}
  while for node pairs $q,q'\in \GG$ the condition was already satisfied in $G$.
  To show that the new decorations form a basis, note first the $g(q)$ for
  $q\in\GG$ already formed a basis of $W$ by assumption of minimality and the
  new decoration $f(p):=y$ is linearly independent.
\end{proof}

\subsection{Extending Arbitrary Symplectic 
Root Systems}\label{sec_extendingArbitrary}

Quite surprisingly, a uniform treatment of the two extremal cases
(nondegenerate and nullspace) is possible and yields an equally strong statement
as in the cases above. For this, one has to introduce an artificial
non-symplectic nondegenerate bilinear form.

\begin{definition}[Mixed Completion]\label{def_MixedCompletion}
  Let $W,\;\langle,\rangle$ be a symplectic vector space over $\k$ and choose
  $\pi:\;W\rightarrow W^\perp$ a fixed projection to the nullspace and $(,)$ a
  nondegenerate symmetric bilinear form on the vector space $W^\perp$.
  We define the {mixed completion} $\langle\langle,\rangle\rangle: W\times
  W\to \k$ by
  $$\langle\langle v,w\rangle\rangle
  :=\langle v,w\rangle+\left(\pi(v),\pi(w)\right)$$
\end{definition}
\begin{lemma}\label{lm_MixedCompletion}
 The mixed completion is a nondegenerate bilinear form on $W$.
\end{lemma}
\begin{proof}
  Suppose some $v\in W$ to be in the radical, i.e. 
  $$\forall_{w\in W}\;\langle\langle v,w\rangle\rangle=
    \langle v,w\rangle+\left(\pi(v),\pi(w)\right)=0$$
  Suppose first $w\in W^\perp$, then the left (symplectic) term vanishes, hence
  for the entire term to be $0$, the second $(,\pi(w))$-term has to vansh as
  well. Since $\pi|_{W^\perp}$ is the identity and $(,)$ is
  nondegenerate,
  we deduce $\pi(v)=0$ i.e. $v\in \ker(\pi)$. But then already for all $w\in W$
  $\left(\pi(v),\pi(w)\right)=0$ and hence the assumption
  reduces to:
  $$\forall_{w\in W}\; \langle v,w\rangle=0\qquad \Rightarrow v\in W^\perp$$
  Because $\pi$ was a projection the two deductions $v\in \ker(\pi)$ and $v\in
  W^\perp$ amount to $v=0$. Hence the radical is $\{0\}$ and $\langle\langle
  ,\rangle\rangle$ is indeed nondegenerate as asserted.
\end{proof}

With this tool we combine the techniques for extending extraspecials as well as
nullspaces (Sections \nref{sec_extendingExtraspecial} and
\nref{sec_extendingNullspace}) by effectively writing the new graph as a
symmetric difference of two graphs obtained by the two methods. Thus we achive
a general extension result for arbitrary minimal symplectic root systems:

\begin{theorem}[Minimal Extensions]\label{thm_Extension}
  Let $\DD$ be a graph, $p\in\DD$ a node and $G=(g,W)$ a
  {minimal} symplectic root system of $\GG:=\DD-p$. Then there
  exists a unique minimal extension $F=(f,V)$ of $G$ to $\DD$. 
\end{theorem}
\begin{remark}
  Restriction (Theorem \nref{thm_Restriction}) shows, that if $G$ is of type
  $(n',k')$, then $F$ is either of type $(n,k)=(n',k'+1)$, which case is similar
  to extraspecial extension Theorem \nref{thm_ExtensionExtraspecial}, or
  $(n,k)=(n'+1,k'-1)$, which combines the approach with a new symplectic
  basepair as in Theorem \nref{thm_ExtensionNullspace}.\\

  Which case applies and which
  precise decoration the new point receives can generally only be decided along
  the steps of the proof below. However, uniqueness opens the path to directly
  write down an extended symplectic root system. Also, the case of a double
  extension of an extraspecial symplectic root system allows for a direct
  treatment; see Lemma \nref{lm_DoubleExtensionExtraspecial}.
\end{remark}
\begin{proof}
  Choose a fixed projection to the nullspace $\pi:\;W\rightarrow W^\perp$,
  a nondegenerate symmetric bilinear form $(,)$ on the vector
  space $W^\perp$ and a  mixed completion of the
  symplectic form on $W$ (Definition \nref{def_MixedCompletion}), which is
  nondegenerate  by Lemma \nref{lm_MixedCompletion}.
    $$\langle\langle v,w\rangle\rangle
  =\langle v,w\rangle+\left(\pi(v),\pi(w)\right)$$
  The proof proceeds at first largly along the Extraspecial Extension Theorem
  \nref{thm_ExtensionExtraspecial}:
  \begin{itemize}
    \item Consider the
      indicator function of the neighbourhood of $p$ ($\subset\GG$):
      $$\lambda:\;\GG\rightarrow \F_2 \qquad 
	\forall_{q\in \GG}\;\lambda(q)=1\;:\Leftrightarrow\; pq\in Edges(\DD)$$
    \item Because $G$ is minimal, again there is a
      unique linear form on $W$ such that
      $$\tilde{\lambda}:\;W\rightarrow \F_2\qquad
	\forall_{q\in \GG}\;\lambda(q)=\tilde{\lambda}(g(q))$$
    \item By nondegeneracy of the mixed completion, we find a unique
      $$\tilde{w}_0\in W\qquad 
	\forall_{w\in W}\;\langle\tilde{w}_0,w\rangle
	=\tilde{\lambda}(\tilde{w})$$
      which we decompose according to the projection $\pi: W\rightarrow
      W^\perp$:
      $$\tilde{w}_0=(\tilde{w}_0-\pi(\tilde{w}_0))+\pi(\tilde{w}_0)
	=:w_0+z_0\in \ker(\pi)\oplus W^\perp$$
  \end{itemize}
      At this point, let us go one step in the proof back by rewriting the
      decomposition      for the linear form $\tilde{\lambda}:W\rightarrow \F_2$
      and even the      indicator function $\lambda:\GG\rightarrow\F_2$:
      \begin{align*}
       \forall_{w\in W}\qquad \tilde{\lambda}(w)
	&=\langle\langle\tilde{w}_0,w\rangle\rangle\\
	&=\langle\langle w_0,w\rangle\rangle
	  +\langle\langle z_0,w\rangle\rangle\\
	&=:\tilde{\lambda}_{W/W^\perp}(w)+\tilde{\lambda}_{W^\perp}(w)\\
	\forall_{q\in \GG}\qquad \lambda(q)
	&=\tilde{\lambda}\left(g(q)\right)\\
	&=\tilde{\lambda}_{W/W^\perp}\left(g(q)\right)
	  +\tilde{\lambda}_{W^\perp}\left(g(q)\right)\\
	&=:\lambda_{W/W^\perp}(w)+\lambda_{W^\perp}(q)
      \end{align*}
      The subindices $W/W^\perp,W^\perp$ were chosen for the following reason:
      \begin{align*}
	\tilde{\lambda}_{W/W^\perp}(w)
	&=\langle\langle w_0,w\rangle\rangle\\
	&=\langle w_0,w\rangle+\left(\pi(w_0),\pi(w)\right)\\
	w_0\in\ker(\pi)\qquad
	&=\langle w_0,w\rangle\\
	&\qquad\Rightarrow \tilde{\lambda}_{W/W^\perp}:
	  W\rightarrow W/W^\perp\rightarrow \F_2\\
	\tilde{\lambda}_{W^\perp}(w)
	&=\langle\langle z_0,w\rangle\rangle\\
	&=\langle z_0,w\rangle+\left(\pi(z_0),\pi(w)\right)\\
	z_0\in W^\perp\qquad
	&=\left(\pi(z_0),\pi(w)\right)\\
	\pi\;projection\qquad
	&=\left(\pi(z_0),\pi(w)\right)\\
	&\qquad\Rightarrow \tilde{\lambda}_{W^\perp}:
	  W\stackrel{\pi}{\rightarrow} W^\perp\rightarrow \F_2
      \end{align*}
      The intuition behind the upcoming construction is the following: The
      indicator functions $\lambda_{W/W^\perp},\lambda_{W^\perp}$ on $\GG$
      define different graphs $\DD_{W/W^\perp},\DD_{W^\perp}$ extending $\GG$.
      \begin{itemize}
	\item The first extension can be realized by an element $\langle
	  w_0,-\rangle$, even though this is not generally true if the
	  symplectic form is degenerate. It can thus be realized by adding
	  $f(p):=w_0\in W$ as in the Extraspecial Extension Theorem	 
	  \nref{thm_ExtensionExtraspecial}. Note that there we added a
	  nullvector $z$ only to again yield a minimal symplectic root
	  system $V=W\oplus z\k$.
	\item  The second extension can be realized by an element $\left(
	  z_0,\pi(-)\right)$, even though it is not generally true that an
	  indicator function $\lambda$ factorizes over $\pi$. It can thus be
	  realized by adding as in the Nullspace Extension Theorem
	  \nref{thm_ExtensionNullspace}
	  {\bf either} $f(p)=0$ for $z_0=0$ (and adding a nullvector to achieve
	  a minimal symplectic root system) {\bf or} for $z_0\neq 0$ by
	  projecting to $\ker(\tilde{\lambda})$ and introducing a new basepair
	  $x,y$.
      \end{itemize}
      The indicator function $\lambda$ for $\DD$ we wish to finally achieve has
      been written above as the $\F_2$-sum of the indicator functions
      $\lambda_{W/W^\perp},\lambda_{W^\perp}$ and hence the neighbourhood of $p$
      in $\DD$ is constructed as a symmetric difference for 
      $\DD_{W/W^\perp},\DD_{W^\perp}$.\\

      We now carry out this proof idea: The precise definition of $V$ and hence
      the type of symplectic root system we construct depends crucially on
      $z_0$:\\

  \begin{enumerate}
    \item For $z_0=0$ and hence $\lambda_{W^\perp}=0$ we proceed as
      in the proof of Theorem
      \nref{thm_ExtensionExtraspecial}: 
      \begin{align*}
	V
	&:= W\oplus^\perp z\k\\
	f(p)
	&:=w_0+z\\
	f(q)
	&:=g(q)\qquad \forall q\in \GG
  \end{align*}
      Note that the vice-versa restriction to $\GG$ is equal to $G$.
      The pair $(f,V)$ defined this way satisfies the axioms of a symplectic
      root system of type $(n,k)=(n',k'+1)$ extending $G$ of type $(n',k')$,
      which we prove as follows:
      Obviously, the new decorations
      generate $V$ and the symplectic root system property for edges
      $qq'\in\GG$  hold already in $G$.\\

      For the new edges $pq$ adjacent to
      $p\in \DD$ the symplectic root system property is fulfilled as follows:
      \begin{align*}
	\langle f(p),f(q)\rangle
	&= \langle w_0+z,g(q)\rangle\\
	z\in V^\perp\qquad
	&=\langle w_0,g(q)\rangle\\
	w_0\in\ker(\pi)\qquad
	&=\langle w_0,g(q)\rangle+\left(w_0,g(q)\right)\\
	&=\langle\langle w_0,g(q)\rangle\rangle\\
	&=\lambda_{W/W^\perp}(q)\\
	\lambda_{W^\perp}=0\qquad 
	&=\lambda(q)
      \end{align*}
      Hence $\langle f(p),f(q)\rangle=0$ iff $pq$ non-adjacent in $\DD$.\\

 \item For $z_0\neq 0$ and hence $\tilde{\lambda}_{W^\perp}\neq 0$ we combine
      the previous approach using $w_0$ with the introduction of a new
      hyperbolic plane $H_{n+1}$ as in the proof of Theorem
      \nref{thm_ExtensionNullspace}: Choose an element $x\in W$ with
      $\tilde{\lambda}_{W^\perp}(x)=1_{\F_2}$; such an $x$ can be used to
      project $W\rightarrow \ker(\tilde{\lambda}_{W^\perp})$ by $w\mapsto
      w-\tilde{\lambda}_{W^\perp}(w)\cdot x$. Then define: 
      \begin{align*}
	V
	&:=\underbrace{\ker\left(\tilde{\lambda}_{W^\perp}\right)}_{V^\perp}
	  \oplus^\perp \left( x\k\oplus y\k\right)=W\oplus y\k
	\qquad \langle x,y\rangle:=1\\
	f(p)
	&:=w_0+y\\
	f(q)
	&:=g(q)\\
	&:=\underbrace{g(q)-\tilde{\lambda}_{W^\perp}\left(g(q)\right)\cdot x}
	  _{\in \ker\left(\tilde{\lambda}_{W^\perp}\right)}
	  +\tilde{\lambda}_{W^\perp}\left(g(q)\right)\cdot x
	    \qquad \forall q\in \GG
      \end{align*}
      Note that the vice-versa restriction to $\GG$ is equal to $G$.
      The pair $(f,V)$ defined this way satisfies the axioms of a symplectic
      root system of type $(n,k)=(n'+1,k'-1)$ extending $G$ of type
      $(n',k')$, which we prove as follows:\\

      The new decorations
      generate $V$ and the symplectic root system property for edges
      $qq'\in\GG$  hold already in $G$. For the new edges $pq$ adjacent to
      $p\in \DD$ the symplectic root system property is fulfilled as follows:
      \begin{align*}
	\langle f(p),f(q)\rangle
	&= \langle w_0+y,g(q) \rangle \\
	&=\langle w_0,g(q)\rangle
	+\langle y,
	  \underbrace{ g(q)-\tilde{\lambda}_{W^\perp}\left(g(q)\right)\cdot x}
	  _{\in \ker\left(\tilde{\lambda}_{W^\perp}\right)}
	  +\tilde{\lambda}_{W^\perp}\left(g(q)\right)\cdot x\rangle\\
	&=\tilde{\lambda}_{W/W^\perp}\left(g(q)\right)
	+\langle y,\tilde{\lambda}_{W^\perp}\left(g(q)\right)\cdot x\rangle\\
	&=\tilde{\lambda}_{W/W^\perp}\left(g(q)\right)
	+\tilde{\lambda}_{W^\perp}\left(g(q)\right)\\
	&={\lambda}_{W/W^\perp}\left(q\right)
	+{\lambda}_{W^\perp}\left(q\right)\\
	&=\lambda(q)
      \end{align*}
      Hence $\langle f(p),f(q)\rangle=0$ iff $pq$ non-adjacent in $\DD$.\\
  \end{enumerate}
\end{proof}

%
%
%

\subsection{Tool: Double-Extensions of Extraspecial Symplectic
Root Systems}\label{sec_doubleExtension}

Finally, we give a practical criterion to obtain the unique minimal symplectic
root system which
are double extension of extraspecial ones without having to use mixed
completions: Let $\DD$ be a graph, $p,q\in\DD$ vertices and $G=(g,W)$ a
{minimal extraspecial} symplectic root system of $\GG:=\DD-\{p,q\}$ (i.e. of
type
 $(n,0)$ for $2n=|\GG|$). We apply Theorem \nref{thm_Extension} twice: The
first extension yields an almost extraspecial extension $(n,1)$, the second
extension allows two cases:
\begin{itemize}
  \item a type $(n,2)$ minimal extension $F=(f,V)$ of $G$ to $\DD$.
  \item an extraspecial minimal extension $F=(f,V)$ of $G$ to $\DD$ 
    (=type $(n+1,0)$)
\end{itemize}
To determine the case and the precise new decorations $f(p),\;f(q)$ we now may
first apply the much easier Extraspecial Extension Theorem
\nref{thm_ExtensionExtraspecial}
to the extensions of $G$ to $\DD-p$ resp. $\DD-q$, yielding elements $w_p\in W$
resp. $w_q\in W$ for the $w_0$ in the proof of
Theorem \nref{thm_ExtensionExtraspecial}. These elements determine the full
extension as follows:

\begin{lemma}[Double-Extension of
Extraspecials]\label{lm_DoubleExtensionExtraspecial}
  Depending on $w_p,w_q$ as defined above the extension $F=(f,V)$ has the
following form
  \begin{center}
  \begin{tabular}{r|cc}
    & $pq\in Edges(\DD)$ & $pq\not\in Edges(\DD)$\\
    \hline
    $w_p\perp w_q$ & $(n+1,0)$ & $(n,2)$ \\
    $w_p\not\perp w_q$ & $(n,2)$ & $(n+1,0)$\\
  \end{tabular}
  \end{center}
  with the previous decorations $g(q)$ on $\GG$ and new decorations
  \begin{center}
  \begin{tabular}{l|lllr}
    Case & $V$ & $f(p)$ & $f(q)$ \\
    \hline
    $(n,2)$ & $W\oplus^\perp z_p\F_2 \oplus^\perp z_q\F_2$ 
      & $w_p+z_p$ & $w_q+z_q$ & \\
    $(n+1,0)$ & $W\oplus^\perp \left(x\F_2 \oplus y\F_2\right)$ 
      & $w_p+x$ & $w_q+y$ & with $\langle x,y\rangle:=1$\\
  \end{tabular}
  
  \end{center}
\end{lemma}
\begin{proof}
  By the Extraspecial Extension Theorem \nref{thm_ExtensionExtraspecial} there
exist unique extensions  $F_p:=(F_p,z_p\F_2\oplus^\perp W)$ resp.
$F_q:=(F_q,z_q\F_2\oplus^\perp W)$ of  $G$ to $\DD_p:=\DD-p$ resp.
$\DD_q:=\DD-q$ with $f_p(q):=z_p+w_p$ resp.  $f_q(p):=z_q+w_q$ (and
$\forall_{r\in \GG}\;f_p(r)=f_q(r):=g(r)$) for   specificly constructed
$w_p,w_q\in W$.\\

  Hence we may aim to construct a symplectic root system $F=(f,V)$ for $\DD$
such that the
  restrictions to $\DD_p$ resp. $\DD_q$ are $F_p$ resp. $F_q$ (conversely
  we again see that every symplectic root system on $\DD$ has to be of this
form).\\

Note first, that any pair of vertices
except $p,q$ is contained
  in $\DD_p$ or $\DD_q$ and hence by construction the symplectic root system
condition has to be
  checked only for the particular pair $pq$. Here we calculate
  \begin{align*}
      \langle f(p),f(q)\rangle
      &=\langle z_q+w_q,z_p+w_p\rangle \\
      &=\langle z_q,z_p\rangle+\langle w_q,w_p\rangle
  \end{align*}
  because by construction $z_p,z_q\perp W$. The expression should be $0,1$
  depending on wheather $pq\in Edges(\DD)$. Hence depending on this and
  $w_p\perp w_q$, precisely one of the possibilities $\langle
  z_p,z_q\rangle=0,1$ turns $F$ into an symplectic root system of $\DD$.\\
\end{proof}

\section{Example: Symplectic Root Systems for ADE}\label{sec_CartanType}

The unique extension in Theorem \nref{thm_ExtensionExtraspecial} requires
the smaller symplectic root system to be extraspecial. To avoid the general
Theorem \nref{thm_Extension} we use in the following proof the double extension
of extraspecials in Lemma \nref{lm_DoubleExtensionExtraspecial}. We start with
an induction to construct unique minimal symplectic root systems
for all $A_{2n}$, that turn out to be all extraspecial (nullity $k=0$). Then we
add either one or two more vertices to achieve the other diagrams with higher
nullity.\\

Extending the extrapecial minimal $A_{2n}$ by one node by invoking
Theorem \nref{thm_ExtensionExtraspecial} does not require any further
consideration and immediately yields almost
extraspecial symplectic root systems (nullity $k=1$) on $A_{2n+1},D_{2n+1},E_7$.
On the other hand, when doubly extending an extraspecial $A_{2n-2}$ to
$A_{2n}$ (induction) and to $D_{2n},E_6,E_8$ using Lemma
\nref{lm_DoubleExtensionExtraspecial} one has to check by explicit calculation,
which case of this theorem applies: usually extrapecial and only for $D_{2n}$ of
type $(n,2)$.

\begin{theorem}\label{thm_Cartan}
  Each simply-laced root system in characteristic 0 of Cartan type 
  $A_n,D_{n\geq 4},E_{6,7,8}$ admits symplectic root systems precisely of the
  following types. Each symplectic root system in the lists exists and is unique
  up to isomorphism. Note that $D_{2n+2}$ has $3$ non-isomorphic
  $(n,1)$-quotients.\\

\begin{center}
\begin{tabular}{llllll}
  $A_{2n}$ & $(n,0)$ minimal 
  & $D_{2n+1}$ & $(n,1)$ minimal
  & $E_6$ & $(3,0)$ minimal\\
  $A_{2n+1}$ & $(n,1)$ minimal 
  & $D_{2n+1}$ & $(n,0)$ quotient
  & $E_7$ & $(3,1)$ minimal\\
  $A_{2n+1}$ & $(n,0)$ quotient
  & $D_{2n+2}$ & $(n,2)$ minimal
  & $E_7$ & $(3,0)$ quotient\\
    &
  & $D_{2n+2}$ & $(n,1)$ $3$ quotients
  & $E_8$ & $(4,0)$ minimal\\
    &
  & $D_{2n+2}$ & $(n,0)$ quotient
  & & \\
\end{tabular}
\end{center}
$\qquad$\newline

In the proof below, explicit decorations are constructed for each case.

\end{theorem}

\begin{remark}\label{rem_ActionOnCartan}
 Graph automorphisms preserve the minimal symplectic root systems and hence
induce a symplectic isomorphism on $V$. E.g. the flip of
$A_{2n}$ correspond to the symplectic involution $\forall_i\;x_i\leftrightarrow
y_i$. The same holds for unique quotients.\\

On the other hand, the $3$ non-isomorphic quotients of type $(n,1)$ for
$D_{2n+2}$ are permuted by the graph automorphisms as follows:
\begin{itemize}
 \item For $n\neq 0$ there is a single order-$2$ graph automorphism on
  $D_{2n+2}$
  interchanging two of the three quotients while preserving the third.
 \item The graph $D_4$ has the exceptional automorphism group $\S_3$, which
  permutes all three quotients.
\end{itemize}
\end{remark}

The proof will proceed in several steps. We first clarify $A_{2n}$ which is the
building block for all other diagrams:
\begin{lemma}\label{lm_symplectic root systemAn}
  $A_{2n}$ has a unique symplectic root system, which is minimal and
  extraspecial. It's precise form is in the natural linear node ordering:
  $$x_{n}+x_{n-1},\;y_{n-1}+y_{n-2},\;x_{n-2}+x_{n-3},\;\ldots,\;
    y_{n-2}+y_{n-3},\;x_{n-1}+x_{n-2},\;y_{n}+y_{n-1}$$
  where the midmost decorations are $\ldots y_2+y_1,\;x_1,\;y_1,\;x_2+x_1\ldots$
  or respectively $\ldots x_2+x_1,\;y_1,\;x_1,\;y_2+y_1\ldots$ depending on the
  genus of $n$ (note that $x,y$ may be switched by the obvious symplectic
  isomorphism)
\end{lemma}
\begin{proof}
  Certainly for $n=1$ the unique symplectic root system $x_1,y_1$ of type
$(1,0)$ is minimal and extraspecial, see example \nref{exm_A2}.\\

  Now assume inductively that $A_{2n}$ has the unique extraspecial minimal
  symplectic  root system $G=(g,W)$ as asserted. We
  wish to invoke Lemma \nref{lm_DoubleExtensionExtraspecial} to extend this
  extraspecial minimal symplectic root system on $A_{2n}$ by new vertices $p$
  resp. $q$ attached to each end of $A_{2n+2}$. To prove that we indeed land in
  the extraspecial case 1 of the Lemma and to determine the precise new
  decoration, we need to first add one node on each end seperately with
  decorations $z_p+w_p$ resp. $z_q+w_q$.\\

  This means we have to calculate an
  actually unique $w_p$ (resp. $w_q$), such that $\langle w_p,g(r) \rangle=1$
  for the  leftmost node $r$ of  $A_{2n}$ and $\langle w_p,g(r') \rangle=0$ for
  all  other vertices. We indeed verify easily, that $w_p=y_{n}$ has this
  property:
  \begin{align*}
    \langle w_p,x_{n}+x_{n-1} \rangle
    &=\langle y_{n}, x_{n}+x_{n-1} \rangle =1\\
    \langle w_p,x_{i}+x_{i-1} \rangle
    &=0 \;\;\forall_{i<n}\\
    \langle w_p,y_{i}+y_{i-1} \rangle
    &=0 \;\;\forall_{i<n}\\
    \langle w_p,y_{n}+y_{n-1} \rangle
    &=\langle y_{n}, y_{n}+y_{n-1}\rangle =0
  \end{align*}
  for $n>1$ respectively for $n=1$:
  \begin{align*}
    \langle w_p,x_1 \rangle
    &=\langle y_{1}, x_1 \rangle =1\\
    \langle w_p,y_1 \rangle
    &=\langle y_{1}, y_1\rangle =0
  \end{align*}  

  By symmetry we analogously have $w_q=x_{n}$. Because $p,q$ are not connected
  and $w_p\not\perp w_q$, Lemma \nref{lm_DoubleExtensionExtraspecial}
  applies yielding case 1. Hence it indeed provides $V$ to be extraspecial with
  new symplectic base-pair $z_p,z_q$, which we from now on call
  $y_{n+1},x_{n+1}$, and the new decorations 
  $$f_{A_{2n+2}}(p)=z_p+w_p=y_{n+1}+y_{n} \qquad
    f_{A_{2n+2}}(q)=z_q+w_q=x_{n+1}+x_{n}$$
  which is as asserted (after a symplectic isomorphism switching all $x_i,y_i$).
\end{proof}

For later use, we also wish to calculate the necessary $w_0$ to attach a node
to certain vertice in $A_{2n}$ in:

\begin{lemma}\label{lm_ws}
  When applying Theorem \nref{thm_ExtensionExtraspecial} to the symplectic
  root system $G=(g,W)$ of $\GG=A_{2n}$ constructed in Lemma \nref{lm_symplectic
  root systemAn}, the following
  explicit elements $w_0\in W$ arrise, depending on the neighbourhood in $\GG$
  of the new to-be-added node $p$:
  \begin{itemize}
    \item $w_0=y_{n}$, for attaching $p$ to only the first node of $A_{2n}$.
    \item $w_0=x_{n}+x_{n-1}$, for attaching $p$ to only the second node of
      $A_{2n}$ ($n\geq 2$) 
    \item $w_0=y_{n}+y_{n-1}+y_{n-2}$, for attaching $p$ to only the third
      node of $A_{2n}$ ($n\geq3$).
  \end{itemize}
  (the excluded cases are as follows: for $n=1$ the second node is the first
  node from the right and for $n=2$ the  third node is the second node from the
  right)
\end{lemma}
\begin{proof}
  The first claim (=attaching to the first node) has already been shown
  during the inductive proof of Lemma \nref{lm_symplectic root systemAn}.\\

  For the second claim (=attaching to the second node), note that the newly
  added $p$ has thus the same neighbourhood as the first node in $A_{2n}$,
  hence also the $w_0$ coincides with the decoration of this first node
  ($x_n+x_{n+1}$ resp. $x_1$ for $n=1$).\\

  For the third claim (=attaching to the third node) we directly calculate,
  that $w_0=y_{n}+y_{n-1}+y_{n-2}$ is  orthogonal on all node decorations
  except the third. Again, all vertices decorated only by $y_i$ or by
  $x_{i<n-2}$ have obviously orthogonal decoration):
\begin{align*}
  \langle y_{n}+y_{n-1}+y_{n-2},x_{n}+x_{n-1}\rangle
  &=\langle y_{n},x_{n} \rangle+\langle y_{n-1},x_{n-1}\rangle
  =0\\
  \langle y_{n}+y_{n-1}+y_{n-2},y_{n-1}+y_{n-2}\rangle
  &=0\\
  \langle y_{n}+y_{n-1}+y_{n-2},x_{n-2}+x_{n-3}\rangle
  &=\langle y_{n-1},x_{n-1}\rangle
  =1\\
  &\cdots\\
  \langle y_{n}+y_{n-1}+y_{n-2},y_{n-2}+y_{n-3}\rangle
  &=0\\
  \langle y_{n}+y_{n-1}+y_{n-2},x_{n-1}+x_{n-2}\rangle
  &=\langle y_{n-1},x_{n-1}\rangle+\langle y_{n-2},x_{n-2}\rangle
  =0\\
  \langle y_{n}+y_{n-1}+y_{n-2},y_{n}+y_{n-1}\rangle
  &=0
\end{align*}
\end{proof}

We may now continue with the cases $D_{2n+1}$ for $n\geq 2$ resp. $E_7$ which
only require to invoke Theorem \nref{thm_ExtensionExtraspecial} on the obvious
subgraph $A_{2n}$ resp. $A_6$ with unique extraspecial symplectic root system by
the preceeding lemma. The
$w_0$'s calculated above yield the precise decorations of the new node $p$
attached to the second resp. third node to be:
\begin{align*}
f_{D_{2n+1}}(p)&=z_1+x_{n}+x_{n-1}\\
f_{E_7}(p)&=z_1+y_{3}+y_{2}+y_{1}
\end{align*}
Finally, as in Lemma \nref{lm_symplectic root systemAn}, further calculations
are needed when doubly extending the unique minimal extraspecial symplectic
root systems on $A_{2n}$ to $D_{2n+2},E_6,E_8$ using
Lemma \nref{lm_DoubleExtensionExtraspecial}:\\

We first check that the extensions to $E_6$ resp. $E_8$ by attaching vertices
$p,q$ to the first and second node of the obvious subgraphs $A_4$ resp. $A_6$
again yield extraspecial symplectic root system. The respective $w_p,w_q$
were determined in Lemma
\nref{lm_ws} to be $y_{n},\;x_{n}+x_{n-1}$. As $w_p\not\perp w_q$ and $p,q$ not
connected we indeed again land in the extraspecial case 1 of the Lemma and
hence the newly added vertices receive the following decorations:\\

\begin{center}
\begin{tabular}{ll}
$f(p):=z_q+w_q=y_{n+1}+y_{n}$ & $\qquad f(q):=z_p+w_p=x_{n+1}+x_{n}+x_{n-1}$\\
$f_{E_6}(p):=y_{3}+y_{2}$ & $\qquad f_{E_6}(q):=x_{3}+x_{2}+x_{1}$\\
$f_{E_8}(p):=y_{4}+y_{3}$ & $\qquad f_{E_8}(q):=x_{4}+x_{3}+x_{2}$\\
\end{tabular}
\end{center}$\qquad$\newline 

Then we turn to extending $D_{2n+2}$ from the
obvious subgraph $A_{2n}$ by attaching two nodes $p,q$ both to the first
node. Again by Lemma \nref{lm_ws} we thus get $w_p=w_q=y_{n}$. Still
$p,q$ are disconnected, but this time $w_p\perp w_q$, hence we land in case
2 of Lemma \nref{lm_DoubleExtensionExtraspecial} yielding a type $(n,2)$
symplectic root system, where we denote the two newly added nullvectors
$z_1:=z_p$ and
$z_2:=z_q$ and hence yield decorations:\\

$$f_{D_{2n+2}}(p)=z_1+y_{n}\qquad f_{D_{2n+2}}(p)=z_2+y_{n}$$

We still need to check, that all three possible quotient symplectic
root systems of type $(n,1)$
are non-isomorphic. The decorations are respectively:\\

\begin{center}
\begin{tabular}{ll}
$f^{(1)}_{D_{2n+2}}(p)=z+y_{n}$ & $\qquad f^{(1)}_{D_{2n+2}}(q)=x_{n}$\\
$f^{(2)}_{D_{2n+2}}(p)=y_{n}$ & $\qquad f^{(2)}_{D_{2n+2}}(q)=z+x_{n}$\\
$f^{(3)}_{D_{2n+2}}(p)=z+y_{n}$ & $\qquad f^{(3)}_{D_{2n+2}}(q)=z+x_{n}$\\
\end{tabular}
\end{center}$\qquad$\newline 

An isomorphism $\phi$ intertwining $\phi\circ f^{(1)}=f^{(2)}$ needs to send
$\phi:z+y_{n}\leftrightarrow z+y_n$ and fix $z$. To keep all other decorations
$y_i+y_{i-1}$ we further conclude $\phi:z+y_{i}\leftrightarrow z+y_i$, which
contradicts the stability of the midmost decoration $y_1$.\\

An isomorphism $\phi$ intertwining $\phi\circ f^{(1)}=f^{(3)}$ (or
symmetrically $\phi\circ f^{(2)}=f^{(3)}$) needs to send different decorations
$y_n,z+y_n$ to the equal decorations $z+y_n,z+y_n$ which is impossible. Hence
the three quotients of type $(n,1)$ are mutually non-isomorphic. Moreover,
there is a unique quotient of type $(n,0)$ with decoration:\\

\begin{center}
\begin{tabular}{ll}
$f^{(0)}_{D_{2n+2}}(p)=y_{n}$ & $\qquad f^{(0)}_{D_{2n+2}}(q)=x_{n}$\\
\end{tabular}
\end{center}$\qquad$\newline

This concludes the proof of Theorem \nref{thm_Cartan}.

\section{Application}

\subsection{The Action of the Coxeter/Weyl-Group}\label{sec_Weyl}

The following notion can be found e.g. in \cite{Hum72} Section 9.2:

\begin{definition}\label{def_rootsystem}
  Let $E,(,)$ be a euclidean vector space over $\k=\R$. A subset of $\RR\subset
  E$ is called a (finite) root system, if the following axioms are satisfied:
  \begin{itemize}
    \item $\RR$ is finite, spans $E$ and $0\not\in \RR$
    \item If $\alpha\in\RR$, then the only multiples of $\alpha$ in $\RR$ are
      $\pm\alpha$.
    \item If $\alpha\in\RR$, then the reflections leaves $\RR$ invariant:
      $$\sigma_\alpha(\beta)
	:=\beta-\frac{2(\alpha,\beta)}{(\alpha,\alpha)}\alpha$$
    \item If $\alpha,\beta\in\RR$ then
      $\frac{2(\alpha,\beta)}{(\alpha,\alpha)}\in\Z$
  \end{itemize}
  Introduce $d_\alpha=\frac{1}{2}(\alpha,\alpha)$. As usual, we will in the
  following  assume $(,)$ to be norma\-lized such that $d_\alpha\in\Z$ and at
  least one $d_\alpha=1$. This implies in particular
  $$(\alpha,\beta)=d_\alpha \frac{2(\alpha,\beta)}{(\alpha,\alpha)}\in\Z$$
\end{definition}

In fact, every root system has a basis $\RR_0$ of simple roots: A basis
means here a basis of the vector space $E$, such that all roots $\alpha$
are integer linear combinations of $\RR_0$, with either all coefficients
positive or negative. The Cartan matrix of $\RR$ is defined as 
$$\forall_{\alpha\neq \beta\in\RR_0}\;\; C_{\alpha,\beta}
  :=\frac{2(\alpha,\beta)}{(\alpha,\alpha)}=d_\alpha^{-1}(\alpha,\beta)
  \qquad C_{\alpha,\alpha}:=2$$
Note that there is different convention regarding the order of the
Cartan matrix indices $\alpha,\beta$; the one above is custom in
quantum groups.\\

Frequently, the Cartan matrix is visualized as a Dynkin
diagram with nodes $\RR_0$ and edges drawn for $C_{\alpha,\beta}\neq 0$. The
Weyl group $\WW$ is the Coxeter group determined by the symmetrized Cartan
matrix $(,)$. The Weyl group $\WW$ is generated by reflections
$\sigma_\alpha,\alpha\in\RR$. Hence it acts on $E$ as isometries and fixes
$\RR$. The classification of finite Lie algebras implies $d_\alpha\in\{1,2,3\}$
with $d_\alpha=1$ for short roots.

\begin{theorem}
  Let $C$ be the Cartan matrix of a root system $\RR$, $\WW$ the
  corresponding Weyl group and $(,),d_\alpha$
  normalized as in Definition \nref{def_rootsystem}. Let  $\GG$ be the graph
  with nodes $\RR_0$ and edges drawn whenever $(\alpha,\beta)\equiv 1\mod 2$. 

  Suppose $F=(f,V)$ is a minimal symplectic root system for $\GG$.
  Then, the decoration $f:\RR_0\rightarrow V$ can be additively extended
  to a map $\tilde{f}:\RR\rightarrow V$. Furthermore we can define a natural
  action of $\WW$ on $V$ by symplectic isomorphisms, such that $\tilde{f}$
  intertwines the $\WW$-actions on $\RR,V$ and hence especially fixes
  $\tilde{f}(\RR)\subset V$.
\end{theorem}
Note: If $\RR$ is simply-laced, then $\GG$ is just the Dynkin diagram of $\RR$.
For $B_n$ it is the Dynkin diagram $A_1\times\cdots\times A_1$, for $C_n$ it is
$A_1\times A_{n-1}$, for $F_4$ it is $A_1\times A_1\times A_2$ and for $G_2$ it
is $A_2$.
\begin{proof}

  The proof is essentially the invariance of $(,)$ under the action
  of $\WW$:
  By the assumed minimality the image of the sets $\GG=\RR_0$ under $f$
  is a basis of $V$. Hence there exists for each simple root $\alpha\in\RR_0$ a
  unique map $\tilde{\sigma}_\alpha:V\rightarrow V$  such that this basis
  $\{f(\beta)\}_{\beta\in\RR_0}$ is mapped accordingly:
  \begin{align*}
   \tilde{\sigma}_\alpha\left(f(\alpha)\right)
    &:=\tilde{f}\left(\sigma_\alpha(\beta)\right)\\
    &=\tilde{f}\left(\beta-C_{\alpha,\beta}\alpha\right)\\
    &=f(\beta)-C_{\alpha,\beta}f(\alpha)
  \end{align*}
  Note that by additive extension, these formulae hold for non-simple roots
  $\alpha,\beta$ as well, but simple reflection of simple roots will suffice
  here. Note also, that especially for $C$ not simply-laced, $\tilde{f}$ does
  not need to be injective, as e.g. $\beta$ and $\beta-2\alpha$ are mapped to
  the same vectors in the vector space $V$.\\

  We yet have to check that the expressions for $\tilde{\sigma}_\alpha$ indeed
  define a symplectic isomorphism, i.e. preserves the symplectic form
  $\langle,\rangle$ on $V$, which we again check on the basis $f(\RR_0)$:
  \begin{align*}
    &\langle f\left(\tilde{\sigma}_\alpha(\beta)\right),
      f\left(\tilde{\sigma}_\alpha(\gamma)\right)\rangle\\
    &=\langle f(\beta)-C_{\alpha,\beta}f(\alpha),
      f(\gamma)-C_{\alpha,\gamma}f(\alpha)\rangle\\ 
    &=\langle f(\beta),f(\gamma)\rangle
    -C_{\alpha,\beta}\langle f(\alpha),f(\gamma)\rangle
    -C_{\alpha,\gamma}\langle f(\beta),f(\alpha)\rangle
    +\langle f(\alpha),f(\alpha)\rangle\\
    &=\langle f(\beta),f(\gamma)\rangle
    - C_{\alpha,\beta}\cdot d_\alpha C_{\alpha,\gamma}
    + C_{\alpha,\gamma}\cdot d_\alpha C_{\alpha,\beta}+0\\
    &=\langle f(\beta),f(\gamma)\rangle
  \end{align*}
  Here we calculated in $\F_2$ and used that by definition 
  $$\langle f(\alpha),f(\beta)\rangle=(\alpha,\beta)
  =d_\alpha C_{\alpha,\beta}=-d_\beta C_{\beta,\alpha}$$
\end{proof}

\subsection{Commutativity Graphs and Nichols
Algebras}\label{sec_CommutativityGraph}

In the application \cite{Len13} of symplectic root systems to Nichols algebras
over finite nilpotent groups $G$, a symplectic root system is used to determine
a generating set of elements with commutators prescribed by a fixed graph
$\GG$. Thereby the commutator map plays the role of the symplectic form. Note
that again we so far only consider commutators of order $2$.

\begin{definition}
  Let $\GG$ be a graph and $G$ a finite group. A decoration $f:\GG\rightarrow
  G$ is said to have $\GG$ as commutativity graph iff
  \begin{itemize}
    \item For $\alpha,\beta\in \GG$ the images $f(\alpha),f(\beta)$ commute
      iff $\alpha,\beta$ are non-adjacent.
    \item The images $f(\GG)$ are a generating system of $G$ 
  \end{itemize}
  We call such a decoration {minimal} iff the generating system $f(\GG)$ is
  minimal, i.e. no proper subset generates all of $G$.
\end{definition}

Suppose we are given a finite group with $[G,G]=\Z_2$. As
usual for $p$-groups we consider the skew-symmetric, isotropic {commutator map
$[,]$} (see \cite{Hup83}):
$$G\times G\stackrel{[,]}{\longrightarrow} [G,G]=\Z_p$$
$$g,h\mapsto [g,h]=ghg^{-1}h^{-1}$$
$$[h,g]=[g,h]^{-1} \quad [g,g]=1$$
Because $[G,G]$ is clearly central, the map is bimultiplicative
(the right-hand-side argument's works analogously):
\begin{align*}
  [g,h][g',h]
  &=(ghg^{-1}h^{-1})(g'hg'^{-1}h^{-1})\\
  &=g(g'hg'^{-1}h^{-1})hg^{-1}h^{-1}\\
  &=gg'hg'^{-1}g^{-1}h^{-1}\\
  &=[gg',h]
\end{align*}
Thus the commutator map factorizes to $V:=G/G^2\cong \F_2^n$: 
$$V\times V \stackrel{\langle,\rangle}{\longrightarrow} \F_2$$

\begin{example}
  Consider the extraspecial groups $G=2_\pm^{2\cdot n+1}$, which are central
  products of $n$ groups $G_i$ each isomorphic to 
  $$2_+^{2\cdot 1+1}=\D_4 \qquad 2_-^{2\cdot 1+1}=\Q_8$$  
  Central product means hereby, that $G_i^2=Z(G_i)=[G_i,G_i]\cong\Z_2$ of all
  factors are identified, especially $G^2=Z(G)=[G,G]\cong\Z_2$.\\

  Then $V=G/G^2=\times_i \Z_2\times\Z_2$ has a basis of elements
  $x_i,y_i$ and the commutator map corresponds to the standard symplectic form
  $\langle x_i,y_i\rangle=\F_2$. Especially extraspecial groups correspond to
  nondegenerate symplectic vector spaces $V,\langle,\rangle$.\\

  Similary,allowing an additional factor $G_0=\Z_4$, the resulting groups
  are known as almost-extraspecial $2_\pm^{2\cdot n+2}$ corresponding
  to a symplectic vector space of type $(n,1)$.
\end{example}

\begin{remark}
  In \cite{Len13} we start with a group-$2$-cocycle $u\in
  Z^2(G/[G,G],\Z_2)$  for to the central extension $G$, corresponding
  to the symplectic form as follows
  $$u(\bar{g},\bar{h})u^{-1}(\bar{h},\bar{g})
  =[g,h]=\langle\bar{g},\bar{h}\rangle$$
\end{remark}

We now show how symplectic root systems can be used to construct generating
sets of $G$ with prescribed commutativity graph $\GG$. In the application
\cite{Len13} the graph $\GG$ is a Dynkin diagram of a Nichols
algebra $\B(M)$ over the abelian group $\Gamma:=G/[G,G]$ and we
construct the covering Nichols algebra $\B(\tilde{M})$ over $G$. The use of a
symplectic root system thereby guarantees that the $G$-decorations in the
covering Nichols algebra commute iff nodes are disconnected, which is a
necessary condition for finite Nichols algebras, see \cite{HS10}, Prop. 8.1. \\

Note that the next Theorem is restricted to $G$ a group of order $2^n$, but
every group with $[G,G]=\Z_2$ can be written $G=G_{odd}\times G_2$, where
$G_{odd}$ is abelian and of odd order, while $G_2$ is a $2$-group.

\begin{theorem}
  Let $G$ be a 2-group with $[G,G]=\Z_2$ and define as above the following
  symplectic vector space:
  $$\left(V,\langle,\rangle\right):=\left(G/G^2,[,]\right)$$
  Let $\GG$ be a
  graph and $F=(f,V)$ a symplectic root system for $\GG$, then any lifts
  $g_\alpha\in G$ of the  decoration elements $\{f(\alpha)\}_{\alpha\in
  \GG}\subset V=G/G^2$ has $\GG$ as commutativity graph. Moreover, the
  lifts form a minimal generating system  of $G$ iff $F$ was a minimal
  symplectic root system.
\end{theorem}
\begin{proof}
    The main assertion follows almost by definition: For any $\alpha,\beta\in
    \GG$ we show that the lifted $G$-decoration $\alpha\mapsto g_\alpha$
    commute iff the $V$-decorations are symplectically orthogonal, which
    happens by the symplectic root system property iff $\alpha,\beta$ are
    non-adjacent in $\GG$:
    \begin{align*}
	[g_\alpha,g_\beta]
	&=\langle \overline{g_\alpha},\overline{g_\beta}\rangle\\
	&=\langle f(\alpha),f(\beta) \rangle
    \end{align*}
    We yet have to prove that the lifted decorations
    $\{g_\alpha\}_{\alpha\in\GG}$ indeed generate $G$, as the
    symplectic root system decorations $\{f(\alpha)\}_{\alpha\in\GG}$ generate
    $V$ and especially a minimally $G$-generating set corresponds
    to a basis of $V$. This is the content of the following much more
    general theorem for $p$-groups and $V=G/\Phi(G)$:
    \begin{theorem}[Burnside Basis Theorem \cite{Hup83}
      Satz III.3.15 (p. 273)]\label{thm_Burnside}
      Every minimally generating set of a 2-group $G$ (no element may be
      omitted) $g_1,\ldots g_n$ consists precisely of $n=\dim_{\F_2}(V)$
      elements for $V:=G/G^2$, whose images in $V$ form a basis. 
    \end{theorem}
\end{proof}

\printbibliography

\end{document}